\documentclass[12pt]{amsart}
\usepackage{amscd}
\usepackage{amsfonts,amssymb,amsmath,amsthm}
\usepackage{url}
\usepackage{enumerate}
\usepackage{bm}
\usepackage{bbm}
\usepackage{time}
\usepackage{geometry}
\usepackage[all]{xy}

\newtheorem{thm}{Theorem}[section]

\newtheorem{lem}[thm]{Lemma}
\newtheorem{prop}[thm]{Proposition}

\theoremstyle{definition}
\newtheorem{defin}[thm]{Definition}
\newtheorem{rem}[thm]{Remark}
\newtheorem{exa}[thm]{Example}

\numberwithin{equation}{section}

\makeatletter
\@namedef{subjclassname@2020}{%
  \textup{2020} Mathematics Subject Classification}
\makeatother

\author[Ayako Itaba]{Ayako Itaba}
\address{
Institute of Arts and Sciences, 
Tokyo University of Science\\
6-3-1 Niijuku, Katsushika-ku, Tokyo, 125-8585, Japan}
\email{itaba@rs.tus.ac.jp}


\newcommand{\Ext}{{\rm Ext}}

\newcommand{\al}{\alpha}
\newcommand{\be}{\beta}

\newcommand{\si}{\sigma}
\newcommand{\la}{\lambda}

\def\Spec{\mathsf{Spec}}
\def\Proj{\mathsf{Proj}}
\def\Specn{\mathsf{Spec}_{{\rm nc}}}
\def\Projn{\mathsf{Proj}_{{\rm nc}}}

\def\s{\sigma}

\def\cA{\mathcal A}
\def\cO{\mathcal O}

\def\NN{\mathbb N}
\def\PP{\mathbb P}

\def\mod{\mathsf{mod}}
\def\grmod{\mathsf{grmod}}
\def\tors{\mathsf{tors}}
\def\tails{\mathsf{tails}}

\def\Aut{{\rm Aut}}

\def\<{\langle}
\def\>{\rangle}
\def\Spec{\mathsf{Spec}}
\def\Proj{\mathsf{Proj}}
\def\Specn{\mathsf{Spec}_{{\rm nc}}}
\def\Projn{\mathsf{Proj}_{{\rm nc}}}
\def\Tails{\mathsf{Tails}}
\def\Mod{\mathsf{Mod}}

\def\s{\sigma}

\def\cA{\mathcal A}
\def\cO{\mathcal O}

\def\NN{\mathbb N}
\def\PP{\mathbb P}

\def\mod{\mathsf{mod}}
\def\grmod{\mathsf{grmod}}
\def\tors{\mathsf{tors}}
\def\tails{\mathsf{tails}}

\def\Aut{{\rm Aut}}

\def\<{\langle}
\def\>{\rangle}

\begin{document}

\title
[Quant. proj. planes and Beilinson alg. of $3$-dim. quant. poly. alg. for Type S']
{Quantum projective planes and Beilinson algebras of $3$-dimensional quantum polynomial algebras for Type S'} 
\begin{abstract}
Let $A=\mathcal{A}(E,\sigma)$ be a $3$-dimensional quantum polynomial algebra
where $E$ is $\mathbb{P}^{2}$ or a cubic divisor in $\mathbb{P}^{2}$, and $\sigma\in \mathrm{Aut}_{k}E$. 
Artin-Tate-Van den Bergh proved that $A$ is finite over its center if and only if 
the order $|\sigma|$ of $\sigma$ is finite. 
As a categorical analogy of their result, the author and Mori showed that 
the following conditions are equivalent; 
(1) $|\nu^{\ast}\sigma^{3}|<\infty$, where $\nu$ is the Nakayama automorphism of $A$. 
(2) The norm $\|\sigma\|$ of $\sigma$ is finite. 
(3) The quantum projective plane $\mathsf{Proj}_{{\rm nc}}A$ is finite over its center. 
In this paper, we will prove for Type S' algebra $A$ that the following conditions are equivalent; 
(1) $\mathsf{Proj}_{{\rm nc}}A$ is finite over its center. 
(2) The Beilinson algebra $\nabla A$ of $A$ is $2$-representation tame. 
(3) The isomorphism classes of simple $2$-regular modules over $\nabla A$ 
are parametrized by $\mathbb{P}^{2}$. 
\end{abstract}
\maketitle

\section{Introduction}
Throughout this paper, we fix an algebraically closed field $k$ of characteristic $0$. 
All algebras and (noncommutative) schemes are defined over $k$. 
We further assume that all (graded) algebras are finitely generated (in degree $1$) over $k$, 
that is, algebras of the form $k\<x_1, \dots, x_n\>/I$ for some (homogeneous) ideal $I\lhd k\<x_1, \dots, x_n\>$ (where $\deg x_i=1$ for every $i=1, \dots, n$). 

In noncommutative algebraic geometry, a quantum polynomial algebra is a basic and important research object, 
which is a noncommutative analogue of a commutative polynomial algebra. 
A quantum projective space is the noncommutative projective scheme associated 
to a quantum polynomial algebra, 
a quantum projective space is also a basic and important research object in this research area. 
At the beginning of noncommutative algebraic geometry, 
Artin-Tate-Van den Bergh \cite{ATV1} found a nice correspondence 
between $3$-dimensional quantum polynomial algebras and geometric pair $(E,\si)$, 
where $E$ is $\mathbb{P}^{2}$ or a cubic divisor in $\mathbb{P}^{2}$, and $\sigma\in \mathrm{Aut}_{k}E$. 
So, this result allows us to write a $3$-dimensional quantum polynomial algebra $A$ as the form $A=\mathcal{A}(E,\sigma)$. 

For a $3$-dimensional quantum polynomial algebra $A=\cA(E, \s)$, 
Artin-Tate-Van den Bergh \cite{ATV2} gave the following geometric characterization 
when $A$ is finite over its center.
\begin{thm}[{\cite[Theorem 7.1]{ATV2}}]
\label{atv2}
Let $A=\mathcal{A}(E,\sigma)$ be 
a $3$-dimensional quantum polynomial algebra. 
Then the following are equivalent:  
\begin{enumerate}[{\rm (1)}]
\item The order $|\sigma|$ of $\sigma$ is finite. 
\item $A$ is finite over its center. 
\end{enumerate}
\end{thm}

To characterize {\it geometric} quantum projective spaces {\it finite over their centers}, 
Mori \cite{Mo2} introduced the following notion: 

\begin{defin}[{\cite[Definition 4.6]{Mo2}}]
For a geometric pair $(E, \s)$ 
where $E\subset \mathbb{P}^{n-1}$ and $\sigma\in {\rm Aut}_{k}E$, 
we define 
$$\Aut_k(\PP^{n-1}, E):=\{\phi|_E\in \Aut_k E\mid \phi\in \Aut_k \PP^{n-1}\},
\text{ and}
$$
$$
\|\sigma \|:={\rm inf}\{i\in \mathbb{N}^{+} \mid 
\sigma^{i}\in \Aut_k (\PP^{n-1}, E)\}, 
$$
which is called \textit{the norm of $\sigma$}. 
\end{defin}

As a categorical analogue of Theorem \ref{atv2} 
the author and Mori \cite{IMo} showed the result as follows: 
\begin{thm}[{\cite[Theorem 3.8]{IMo}}]
If $A=\cA(E, \s)$ is a $3$-dimensional quantum polynomial algebra with the Nakayama automorphism $\nu\in \Aut A$, then $||\s||=|\nu^*\s^3|$, so the following are equivalent:
\begin{enumerate}[{\rm (1)}]
\item{} $|\nu^*\s^3|<\infty$. 
\item{} $||\s||<\infty$. 
\item{} $\Projn A$ is finite over its center. 
\end{enumerate} 
\end{thm}

In this paper, 
we will prove the following result for Type S' algebra $A=\mathcal{A}(E,\si)$, 
where $E\subset \mathbb{P}^{2}$ is a union of a line and a conic meeting at two points, 
and $\si\in \mathrm{Aut}_{k}E$. 
\begin{thm}[Theorem \ref{Main-I}]
For a $3$-dimensional quantum polynomial algebra $A$ of Type S', 
the following are equivalent: 
\begin{enumerate}[{\rm (1)}]
\item $\Projn A$ is finite over its center. 
\item The Beilinson algebra $\nabla A$ of $A$ is $2$-representation tame 
in the sense of Herschend-Iyama-Oppermann \cite{HIO}. 
\item The isomorphism classes of simple $2$-regular modules over $\nabla A$ are parameterized by $\PP^2$.
\end{enumerate}
\end{thm}
Note that these equivalences are shown for Type S algebra $A=\mathcal{A}(E,\si)$
in \cite [Theorem 4.17, Theorem 4.21]{Mo2}, 
where where $E\subset \mathbb{P}^{2}$ is a triangle, 
and $\si\in \mathrm{Aut}_{k}E$. 


This paper is organized as follows: 
In Section 2, we recall the definitions of a quantum polynomial algebra, a geometric algebra and quantum projective spaces. Also, we describe the characterization of quantum polynomial algebras finite over their centers. 
In Section 3, in order to prove our main results, we calculate the center of Type S' algebra $A$
(Proposition \ref{prop-S'}). 
Finally, in Section 4, we prove our main results (Theorem \ref{thm1}, Theorem \ref{thm2}, Theorem \ref{Main-I}). 


\section{Preliminaries}

\subsection{Quantum polynomial algebras, geometric algebras and quantum projective spaces}
First, we recall the definitions of quantum polynomial algebras and geometric algebras. 

\begin{defin}
A right noetherian graded algebra $A$ is called a 
{\it $d$-dimensional quantum polynomial algebra} if 
\begin{enumerate}[{\rm (1)}]
\item{} $\operatorname{gldim} A=d$, 
\item{} (\textit{Gorenstein condition}) $\Ext^i_A(k, A)\cong \begin{cases}  k & \textnormal { if } i=d, \\
0 & \textnormal { if }  i\neq d, \end{cases}$ 
and 
\item{} $H_A(t):=\sum_{i=0}^{\infty}(\dim _k A_i)t^i=(1-t)^{-d}$ \quad
($H_{A}(t)$ is called \textit{the Hilbert series of $A$}).
\end{enumerate}
\end{defin}
Note that a quantum polynomial algebra is a non-commutative analogue of 
a commutative polynomial algebra. 

\begin{defin}[{\cite[Definition 4.3]{Mo}}]
\label{geometric}
A geometric pair $(E,\sigma)$ consists of a projective scheme $E \subset \PP^{n-1}$ and $\sigma \in {\rm Aut}_{k}\,E$. 
For a quadratic algebra $A=k\<x_1, \dots, x_n\>/I$ where $I\lhd k\<x_1, \dots, x_n\>$ is a homogeneous ideal generated by elements of degree $2$, we define 
$$
\mathcal{V}(I_2):=\{ (p,q)\in\mathbb{P}^{n-1}\times \PP^{n-1}
\,|\,f(p,q)=0\,\,{\rm for\,\,any\,\,} f \in I_2 \}.
$$
\begin{enumerate}[{\rm (1)}]
\item We say that {\it $A$ satisfies {\rm (G$1$)}} if there exists a geometric pair $(E,\sigma)$ such that
$$
\mathcal{V}(I_2)=\{ (p,\sigma(p)) \in \mathbb{P}^{n-1}\times \PP^{n-1}
\,|\,p \in E \}.
$$
In this case, we write $\mathcal{P}(A)=(E,\sigma)$, 
and call $E$ the {\it point scheme} of $A$. 
\item We say that {\it $A$ satisfies {\rm (G$2$)}} if there exists a geometric pair $(E,\sigma)$ such that
$$
I_2=\{ f \in k\<x_1, \dots, x_n\>_2
\,|\,f(p,\sigma(p))=0\,\,{\rm for\,\,any\,\,} p \in E \}.
$$
In this case, we write $A=\mathcal{A}(E,\sigma)$.
\item A quadratic algebra $A$ is called {\it geometric} if $A$ satisfies both (G1) and (G2)
with $A=\mathcal{A}(\mathcal{P}(A))$.
\end{enumerate}
\end{defin}

We remark that a $3$-dimensional quantum polynomial algebra 
is exactly the same as a $3$-dimensional quadratic AS-regular algebra. 
\begin{thm}[\cite{ATV1}]
\label{ATV1}
Every $3$-dimensional quantum polynomial algebra is a geometric algebra where the point scheme is either the projective plane $\PP^2$ or a cubic divisor in $\PP^2$. 
\end{thm}

The type of a $3$-dimensional quantum polynomial algebra $A=\cA(E, \s)$ is defined 
in terms of the point scheme $E\subset \PP^2$ 
(for details, see \cite[Subsection 2.3]{IMa1}): 
\begin{description}
\item[{\rm Type P}] 
      $E$ is $\mathbb{P}^{2}$. 
\item [{\rm Type S}] 
      $E$ is a triangle. 
\item[{\rm Type S'}] 
      $E$ is a union of a line and a conic meeting at two points. 
\item[{\rm Type T}]
      $E$ is a union of three lines meeting at one point. 
\item[{\rm Type T'}]
      $E$ is a union of a line and a conic meeting at one point. 
\item[{\rm Type NC}]
          $E$ is a nodal cubic curve. 
\item[{\rm Type CC}] $E$ is a cuspidal cubic curve. 
\item[{\rm Type TL}] $E$ is a triple line. 
\item[{\rm Type WL}]$E$ is a union of a double line and a line. 
\item[{\rm Type EC}] $E$ is an elliptic curve. 
\end{description}

For a $3$-dimensional quantum polynomial algebra 
$A=\mathcal{A}(E,\sigma)$,  
Artin-Tate-Van den Bergh \cite{ATV2} 
gave the criterion whether $A=\mathcal{A}(E,\sigma)$ 
is finite over its center or not by using geometric data. 

\begin{thm}[{\cite[Theorem 7.1]{ATV2}}]
\label{thm_{ATV2}}
Let $A=\mathcal{A}(E,\sigma)$ be 
a $3$-dimensional quantum polynomial algebra. 
Then the following are equivalent: 
\begin{enumerate}[{\rm (1)}]
\item The order $|\sigma|$ of $\sigma$ is finite. 
\item $A$ is finite over its center. 
\end{enumerate}
\end{thm}

Noncommutative affine and projective schemes were introduced by Artin-Zhang \cite {AZ}. 
Before recalling the definitions of noncommutative affine and projective schemes, 
we start to recall the following definitions: 
\begin{defin} 
{\it A noncommutative scheme {\rm (}over $k${\rm )}} is a pair $X=(\mod X, \cO_X)$ 
consisting of a $K$-linear abelian category $\mod X$ and an object $\cO_X\in \mod X$. 
We say that {\it two noncommutative schemes $X=(\mod X, \cO_X)$ 
and $Y=(\mod Y, \cO_Y)$ are isomorphic}, denoted by $X\cong Y$, 
if there exists an equivalence functor $F:\mod X\to \mod Y$ such that $F(\cO_X)\cong \cO_Y$. 
\end{defin} 



\begin{defin} 
If $R$ is a right noetherian algebra, then we define {\it the noncommutative affine scheme associated to $R$} by $\Specn R=(\mod R, R)$ where $\mod R$ is the category of finitely generated right $R$-modules and $R\in \mod R$ is the regular right module. 
\end{defin} 


\begin{defin}[\cite{AZ}] 
If $A$ is a right noetherian graded algebra, $\grmod A$ is the category of finitely generated graded right $A$-modules, and $\tors A$ is the full subcategory of $\grmod A$ consisting of finite dimensional modules over $k$, then we define {\it the noncommutative projective scheme associated to $A$} by $\Projn A=(\tails A, \pi A)$ where $\tails A:=\grmod A/\tors A$ is the quotient category, $\pi:\grmod A\to \tails A$ is the quotient functor, and $A\in \grmod A$ is the regular graded right module. If $A$ is a $d$-dimensional quantum polynomial algebra, then we call $\Projn A$ {\it a quantum $\PP^{d-1}$}.  In particular,  if $d=3$, then we call $\Projn A$ {\it a quantum projective plane}.   
\end{defin}


Here, we recall the definition of points of a noncommutative scheme. 

\begin{defin}
Let $R$ be an algebra. 
{\it A point of $\Specn R$} is an isomorphism class of a simple right $R$-module $M\in \mod R$ such that $\dim_k M<\infty$. 
A point $M$ is called {\it fat} if $\dim _k M>1$. 
\end{defin} 

\begin{defin} Let $A$ be a graded algebra.  
{\it A point of $\Projn A$} is an isomorphism class of a simple object of the form $\pi M\in \tails A$ where $M\in \grmod A$ is a graded right $A$-module such that $\lim_{i\to \infty}\dim _kM_i<\infty$.  A point $\pi M$ is called {\it fat} if $\lim _{i\to \infty}\dim _k M_i>1$, and, in this case, $M$ is called {\it a fat point module over $A$}.   
\end{defin} 

\begin{rem} 
\begin{enumerate}[(1)]
\item If $A$ is a graded commutative algebra and $p\in \Proj A$ is a closed point, then $\pi (A/\frak m_p)\in \tails A$ is a point where $\frak m_p$ is the homogeneous maximal ideal of $A$ corresponding to $p$.  In fact, this gives a bijection between the set of closed points of $\Proj A$ and the set of points of $\Projn A$.  In this commutative case, there exists no fat point.  
\item It is unclear that fat points are preserved under isomorphisms of $\Projn A$ in general. However, fat point modules are preserved under graded Morita equivalences, so if $A$ and $A'$ are both $3$-dimensional quantum polynomial algebras such that $\Projn A\cong \Projn A'$, then there exists a natural bijection between the set of fat points of $\Projn A$ and that of $\Projn A'$ by Lemma \ref{lem.AOU}.
\end{enumerate}
\end{rem} 

For the reason that the property that $A$ is finite over its center 
is not preserved under isomorphisms of noncommutative projective schemes $\Projn A$, 
we introduced the following definition in \cite{IMo}.  

\begin{defin}[{\cite[Definition 2.6]{IMo}, cf. \cite[Definition 4.14]{Mo2}}]
\label{def_{Z(proj)}}
Let $A$ be a $d$-dimensional quantum polynomial algebra. 
We say that $\mathsf{Proj}_{{\rm nc}}A$ is \textit{finite over its center} 
if there exists a $d$-dimensional quantum polynomial algebra $A'$ 
finite over its center such that 
$\Projn A\cong \Projn A'$.  
\end{defin}

Note that, for a $3$-dimensional quantum polynomial algebra, 
Definition \ref{def_{Z(proj)}} coincides with the original definition in \cite[Definition 4.14]{Mo2} by \cite[Corollary A.10]{AOU} as follows:  

\begin{lem}[{\cite[Corollary A.10]{AOU}}]
\label{lem.AOU}
Let $A$ and $A'$ be $3$-dimensional quantum polynomial algebras. 
Then $\grmod A\cong \grmod A'$ if and only if $\Projn A\cong \Projn A'$. 
\end{lem}

\subsection{Characterization of quantum polynomial algebras finite over their centers}

To characterize \lq\lq geometric\rq\rq quantum projective spaces finite over their centers, 
the following notion was introduced in \cite{Mo2}: 

\begin{defin}[{\cite[Definition 4.6]{Mo2}}]
For a geometric pair $(E, \s)$ 
where $E\subset \mathbb{P}^{n-1}$ and $\sigma\in {\rm Aut}_{k}E$, 
we define 
$$\Aut_k(\PP^{n-1}, E):=\{\phi|_E\in \Aut_k E\mid \phi\in \Aut_k \PP^{n-1}\},
\text{ and}
$$
$$
\|\sigma \|:={\rm inf}\{i\in \mathbb{N}^{+} \mid 
\sigma^{i}\in \Aut_k (\PP^{n-1}, E)\}, 
$$
which is called \textit{the norm of $\sigma$}. 
\end{defin}
For a geometric pair $(E, \s)$, clearly $\|\sigma\| \leq |\sigma|$. 
The following are the basic properties of the norm $\|\s\|$ of $\sigma$.  

\begin{prop}[{\cite [Lemma 2.5]{MU1}, \cite[Lemma 4.16 (1)]{Mo2}}]
\label{lem_Mo2}
Let $A$ and $A'$ be $d$-dimensional quantum polynomial algebras 
satisfying {\rm (G1)} with 
$\mathcal{P}(A)=(E,\sigma)$ and $\mathcal{P}(A')=(E',\sigma')$. 
\begin{enumerate}[{\rm (1)}]
\item{} If $A\cong A'$, then $E\cong E'$ and $|\s|=|\s'|$. 
\item{} If ${\rm grmod}\,A\cong {\rm grmod}\,A'$, then $E\cong E'$ and $||\s||=||\s'||$. 
\end{enumerate}
In particular, if $A$ and $A'$ are $3$-dimensional quantum polynomial algebras such that $\Projn A\cong \Projn A'$, then $E\cong E'$ 
{\rm (}that is, $A$ and $A'$ are of the same type{\rm )} and $||\s||=||\s'||$.  
\end{prop}

At the end of this subsection, we prepare to describe the purpose of this paper. 
For a quantum polynomial algebra, a Calabi-Yau quantum polynomial algebra as follows is easier to handle.

\begin{defin}
A quantum polynomial algebra $A$ is called {\it Calabi-Yau} 
if the Nakayama automorphism of $A$ is the identity.
\end{defin}   

The following lemma claims that every quantum projective plane has a $3$-dimensional Calabi-Yau quantum polynomial algebra as a homogeneous coordinate ring. 

\begin{lem}[{\cite[Theorem 4.4]{IMa2}}]
\label{IMa-Main2}
For every 
$3$-dimensional quantum polynomial algebra $A$,
there exists a $3$-dimensional Calabi-Yau quantum polynomial algebra 
$A'$
such that $\grmod A\cong \grmod A'$
so that $\Projn A\cong \Projn A'$. 
\end{lem}

The following results are main theorems in \cite{IMo}. 

\begin{lem}[{\cite[Theorem 3.4]{IMo}}]
\label{q-main}
If $A'=\cA(E, \s')$ is a $3$-dimensional Calabi-Yau quantum polynomial algebra, then $||\s'||=|{\s'}^3|$, so the following are equivalent:
\begin{enumerate}[{\rm (1)}]
\item{} $|\s'|<\infty$. 
\item{} $||\s'||<\infty$. 
\item{} $A'$ is finite over its center. 
\item{} $\Projn A'$ is finite over its center. 
\end{enumerate}
\end{lem} 
%
\begin{lem}[{\cite[Theorem 3.8]{IMo}}]
\label{q.nu}
If $A=\cA(E, \s)$ is a $3$-dimensional quantum polynomial algebra with the Nakayama automorphism $\nu\in \Aut A$, then $||\s||=|\nu^*\s^3|$, so the following are equivalent:
\begin{enumerate}[{\rm (1)}]
\item{} $|\nu^*\s^3|<\infty$. 
\item{} $||\s||<\infty$. 
\item{} $\Projn A$ is finite over its center. 
\end{enumerate}  
Moreover, if $A$ is of Type T, T', CC, TL, WL, then $A$ is never finite over its center. 
\end{lem}

\begin{lem}[{\cite[Theorem 1.3]{IMo}}]
\label{thm.main-IMo}
Let $A=\cA(E, \s)$ be a $3$-dimensional quantum polynomial algebra such that $E\neq \PP^2$, and $\nu\in \Aut A$  the Nakayama automorphism of $A$. 
Then $||\s||=|\nu^*\s^3|$, so 
the following are equivalent:
\begin{enumerate}[{\rm (1)}]
\item{} $|\nu^*\s^3|<\infty$. 
\item{} $\|\s\|<\infty$. 
\item{} $\Projn A$ is finite over its center.
\item{} $\Projn A$ has a fat piont. 
\end{enumerate} 
\end{lem} 

Note that if $E=\mathbb P^2$, then $||\sigma||=1$, but ${\rm Proj_{nc}}A$ has no fat point. 
(see \cite[Lemma 2.14]{IMo}). 
%

\begin{lem}[{\cite[Corollary 4.2]{IMo}}]
\label{cor.main-IMo} 
Let $A=\cA(E, \s)$ be a $3$-dimensional quantum polynomial algebra with the Nakayama automorphism $\nu\in \Aut A$.  Then the following are equivalent:
\begin{enumerate}[{\rm (1)}]
\item{} $|\nu^*\s^3|=1$ or $\infty$. 
\item{} $\Projn A$ has no fat point.
\item{} The isomorphism classes of simple $2$-regular modules 
over the Beilinson algbera $\nabla A$ of $A$ are parameterized by the set of closed points of $E\subset \PP^2$.  
\end{enumerate}
In particular, if $A$ is of Type P, T, T', CC, TL, WL, then $A$ satisfies all of the above conditions. 
\end{lem}

The Beilinson algebra is a typical example of $(d-1)$-representation infinite algebra by Minamoto-Mori \cite[Theorem 4.12]{MM}. 
Here, for a $d$-dimensional quantum polynomial algebra $A$, 
we define {\it the Beilinson algebra of $A$} by 
$$\nabla A:=\begin{pmatrix} A_0 & A_1 & \cdots & A_{d-1} \\
0 & A_0 & \cdots & A_{d-2} \\
\vdots & \ddots & \vdots & \vdots \\
0 & 0 & \cdots & A_0 \end{pmatrix}.$$

\begin{defin}[{\cite[Definition 6.10]{HIO}, cf. \cite[Definition 4.9]{Mo2}}]
\label{def-d-rep-tame}
We say that a $d$-representation infinite algebra $R$ is {\it $d$-representation tame} 
if the preprojective algebra $\Pi (R)$ of $R$ is right noetherian and  finite over its center. 
\end{defin}

\begin{rem}
\label{rem-IMo1}
We recall from \cite[Remark 4.4]{IMo}. 
For a 3-dimensional quantum polynomial algebra $A$, we expect that the following are equivalent: 
\begin{enumerate}
\item{} $\Projn A$ is finite over its center. 
\item{} The Beilinson algebra $\nabla A$ of $A$ is $2$-representation tame. 
\item{} The isomorphism classes of simple $2$-regular modules over $\nabla A$ are parameterized by $\PP^2$.
\end{enumerate}
In fact, these equivalences are shown for Type S  
in Theorem \ref{Mo2-thm4.17} and Theorem \ref{Mo2-thm4.21} as describe below.  
Our aim of this paper is to prove these equivalences for Type S' (see Theorem \ref{Main-I} in Section 4). 
\end{rem}

\begin{thm}[{\cite[Theorem 4.17]{Mo2}}]
\label{Mo2-thm4.17}
Let $A=\cA(E,\si)$ be a $3$-dimensional quantum polynomial algebra of Type S. 
Then the following are equivalent: 
\begin{enumerate}[{\rm (1)}]
\item $\|\si\|<\infty$. 
\item $\Projn A$ is finite over its center. 
\item The Beilinson algebra $\nabla A$ of $A$ is $2$-representation tame. 
\end{enumerate}
\end{thm}

\begin{thm}[{\cite[Theorem 4.21]{Mo2}}]
\label{Mo2-thm4.21}
Let $A=\cA(E,\si)$ be a $3$-dimensional quantum polynomial algebra of Type S. 
\begin{enumerate}[{\rm (1)}]
\item If the Beilinson algebra $\nabla A$ of $A$ is not $2$-representation tame, 
then the isomorphism classes of simple $2$-regular modules over $\nabla A$ 
are parametrized by the set of points of $E\subsetneq \PP^{2}$. 
\item If the Beilinson algebra $\nabla A$ of $A$ is $2$-representation tame, 
then the isomorphism classes of simple $2$-regular modules over $\nabla A$ 
are parametrized by the set of points of $\PP^{2}$. 
\end{enumerate}
\end{thm}

To prove our main result, we need the following lemmas in \cite{Mo2}: 

\begin{lem}[{\cite[Theorem 4.8]{Mo2}}]
\label{Mo2-thm4.8}
Let $A=\cA(E,\si)$ be a $3$-dimensional quantum polynomial algebra. 
If $\|\si\|=1$ or $\|\si\|=\infty$, then 
the isomorphism classes of simple $2$-regular modules over $\nabla A$ 
are parametrized by the set of points of $E$.
\end{lem}

Let $A$ be a graded algebra and $r\in \mathbb{N}^{+}$. 
We recall that \textit{the $r$th Veronese algebra of $A$} 
is a graded algebra defined by $A^{(r)}:=\bigoplus_{i\in \mathbb{Z}}A_{ri}$. 
Moreover, in \cite[Definition 3.7]{Mo3}, \textit{the $r$th quasi-Veronese algebra of $A$} 
is a graded algebra defined by 
$$
A^{[r]}=
\begin{pmatrix}
A^{(r)} & A(1)^{(r)} & \cdots & A(r-1)^{(r)} \\
A(-1)^{(r)} & A^{(r)} & \cdots & A(r-2)^{(r)} \\
\vdots & \vdots &\ddots &\vdots \\
A(-r+1)^{(r)} & A(-r+2)^{(r)} & \cdots & A^{(r)}
\end{pmatrix}
$$
with the multiplivcation $(a_{ij})(b_{ij}):=(\sum_{k-0}^{r-1}a_{kj}b_{ik})$. 
\begin{lem}[{\cite[Lemma 4.10]{Mo2}}]
\label{Mo2-lem4.10}
For a graded algebra $A$ and $r\in \mathbb{N}^{+}$, 
$Z(A^{[r]})\cong Z(A)^{(r)}$ as graded algebra. 
\end{lem}
\begin{lem}[{\cite[Proposition 4.11]{Mo2}}]
\label{Mo2-prop4.11}
Let $A$ be a right noetherian graded algebra and $r\in \mathbb{N}^{+}$. 
If $Z(A)$ is noetherian and $A$ is finite over its center $Z(A)$, 
then $A^{[r]}$ is right noetherian and finite over its center $Z(A^{[r]})$.  
\end{lem}

If $A$ is a commutative $\NN$-graded algebra and $u\in A$ is a homogeneous element of 
positive degree, then  $\Proj\,A \cong \Proj\,A/(u)\,\bigsqcup\,\Spec\, A[u^{-1}]_{0}$. 
A similar decomposition is known to hold for the noncommutative version. 

\begin{lem}[{\cite[Theorem 4.20]{Mo2}}]
\label{Mo2-thm4.20}
Let $A$ be a right noetherian connected graded algebra satisfying {\rm (H)} in \cite[Definition 3.7]{Mo2}, $u\in A$ a homogeneous regular normalizing element of positive degree 
{\rm (}i.e. $uA=Au${\rm )}, 
$f:\,A\rightarrow A/(u)$ the natural surjection, and $g:\,A\rightarrow A[u^{-1}]$ the natural embedding. 
If $A[u^{-1}]$ is strongly graded, then the functors
\begin{align*}
&\Tails\,A/(u) \rightarrow \Tails\, A;\,\,\pi_{A/(u)}M\mapsto \pi_{A}f_{\ast}M, \\
&\Mod\,A[u^{-1}]_{0}\rightarrow \Tails\,A;\,N\mapsto \pi_{A}g_{\ast}(N\otimes_{A[u^{-1}]_{0}}A[u^{-1}])
\end{align*}
give a bijection from $|\Projn\, A/(u)|\,\bigsqcup\,\Specn\, A[u^{-1}]_{0}$ to 
$|\Projn\,A|$. 
Moreover, in this correspondence, the fat points correspond to fat points. 
\end{lem}

\section{Centers of $3$-dimensional Calabi-Yau quantum polynomial algebras for Type S'}

%
%
In this section, to prove our main theorem in this paper, 
we will calculate the center of a $3$-dimensional Calabi-Yau quantum polynomial algebra for Type S'. 
In order to calculate this, 
we recall the result foy Type S by \cite{Mo1}. 

%
\begin{exa}
\label{ex-TypeS}
Let $A'=\mathcal{A}(E,\sigma')=k\langle x,y,z\rangle/(g_{1},g_{2},g_{3})$
be a $3$-dimensional Calabi-Yau quantum polynomial algebra of Type S, 
where 
$\left\{
\begin{array}{ll}
g_{1}=yz-\alpha zy,\\
g_{2}=zx-\alpha xz,\\
g_{3}=xy-\alpha yx\quad (\alpha^{3}\neq 0,1).
\end{array}
\right.
$
Then $g:=xyz\in Z(A')_{3}$. 
If $A'$ is finite over its center $Z(A')$, 
then $Z(A')=k[x^{n},y^{n},z^{n},g]$ by \cite[Theorem 3.3]{Mo1} where $n=|\sigma|$. 
On the other hand, 
if $A'$ is not finite over its center $Z(A')$, 
then $Z(A')=k[g]$ by (the proof of) \cite[Theorem 3.3]{Mo1}. 
\end{exa}

%

\begin{prop}
\label{prop-S'}
Let $A'=\mathcal{A}(E,\sigma')=k\langle x,y,z\rangle/(g_{1},g_{2},g_{3})$
be a $3$-dimansional Calabi-Yau quantum polynomial algebra of Type S', 
where 
$\left\{
\begin{array}{ll}
g_{1}=yz-\alpha zy+x^{2},\\
g_{2}=zx-\alpha xz,\\
g_{3}=xy-\alpha yx\quad (\alpha^{3}\neq 0,1).
\end{array}
\right.
$
Define $g:=xyz + (1-\alpha^3)^{-1} x^3 \in {A'}_{3}$. 

\begin{enumerate}[{\rm (1)}]
\item
If $A'$ is finite over its center $Z(A')$
{\rm (}that is, $|\alpha|$ is finite{\rm )}, 
then there exists a basis $x$, $y$, $z$ for ${A'}_{1}$ such that 
$
Z(A')=
k[x^{|\alpha|},y^{|\alpha|},z^{|\alpha|},g] 
$. 

\item
If $A'$ is not finite over its center $Z(A')$
{\rm (}that is, $|\alpha|$ is infinite{\rm )}, 
then 
$Z(A')=[g]$.
\end{enumerate}
\end{prop}
\begin{proof}
First, we recall the followings from \cite[Table 1 in Lemma 3.2]{IMo}; 
\begin{itemize}
\item 
$
E=\mathcal{V}(x)\cup \mathcal{V}(x^2-\lambda yz)
\left(\lambda:=\dfrac{\alpha^{3}-1}{\alpha}\right)$,\, 
\item 
$\left\{
\begin{array}{ll}
\sigma'(0,b,c)=(0,b, \alpha c), \\
\sigma'(a,b,c)=(a,\alpha b,\alpha^{-1} c).\\
\end{array}
\right.$
\end{itemize}
Since 
$$
\begin{cases}
(\sigma')^i(0,b,c)=(0,b, \alpha^i c), \\
(\sigma')^i(a,b,c)=(a,\alpha^i b,\alpha^{-i} c)=(\alpha^{-i}a, b, \alpha^{-2i}c),\\
\end{cases}
$$
$(\sigma')^i\in {\rm Aut}_k(\mathbb{P}^2, E)$ if and only if 
$\alpha^{3i}=1$, so $||\sigma'||=|\alpha^3|=|(\sigma')^3|$ 
(\cite[proof of Theorem 3.3]{IMo}). 
Therefore, by Lemma \label{q.main}, 
$||\sigma'||=|(\sigma')^3|=|\alpha^3|<\infty$ if and only if $A'$ is finite over its center 
if and only if $\mathsf{Proj}_{{\rm nc}}\,A'$ is finite over its center. 

We will find generators of the center of $A'$. 
By induction on non-negative integers $i$, $j$ and $k$, 
the following equations are hold; 
\begin{itemize}
\item
$y x^i = \alpha^{-i} x^i y$, $y^j x = \alpha^{-j} x y^j$,
\item
$z x^i = \alpha^{i} x^i z$, $z^k x = \alpha^k x z^k$,

\item
$z y^j = \alpha^{-j} y^j z + 
\displaystyle \alpha^{-3j+2} \left( \sum_{\nu=0}^{j-1} \alpha^{3\nu} \right) 
x^2 y^{j-1}$,
where we define 
$\displaystyle \sum_{\nu=0}^{j-1} \alpha^{3\nu}  = 0$
when $j=0$,

\item
$z^k y = \alpha^{-k} y z^k + 
\displaystyle \alpha^{-k} \left( \sum_{\nu=0}^{k-1} \alpha^{3\nu} \right) 
x^2 z^{k-1}$. 
\end{itemize}
So, if $|\alpha|$ is finite, then
$\displaystyle \sum_{\nu=0}^{|\alpha|-1} \alpha^{3\nu} =0$, 
hence 
$x^{|\alpha|}$, 
$y^{|\alpha|}$ and $z^{|\alpha|}$
are elements of the center of $A'$.
By calculation, we see that, 
whether $|\alpha|$ is finite or not, 
$g:=xyz + (1-\alpha^3)^{-1} x^3$ is also an element of the center of $A'$. 

Suppose that 
$X=\displaystyle \sum_{i+j+k=d} \lambda_{i,j,k} x^i y^j z^k$
is an element of the center of $A'$.
Since
\[
\displaystyle \sum_{i+j+k=d} \lambda_{i,j,k} x^{i+1} y^j z^k 
= x X = X x = \displaystyle \sum_{i+j+k=d} 
\alpha^{k-j} \lambda_{i,j,k} x^{i+1} y^j z^k ,
\]
$
\lambda_{i,j,k} = 
\alpha^{k-j} \lambda_{i,j,k}$
holds for each $i$, $j$ and $k$. 
So it follows that 
\begin{description}
\item[(i)]
if $|\alpha|$ is finite and $\lambda_{i,j,k}\neq 0$, 
then 
$k-j$ is divided by $|\alpha|$, 
and if $|\alpha|$ is infinite and $\lambda_{i,j,k}\neq 0$, 
then $j=k$. 
\end{description}
Moreover, by calculation, 
\[
\begin{array}{r@{\ }c@{\ }l} 
\displaystyle \sum_{i+j+k=d}  
\lambda_{i,j,k} x^i y^j z^{k+1} & = & Xz = zX \\
& = & \displaystyle
\sum_{i+j+k=d}
\left(
\alpha^{i-j} \lambda_{i,j,k} + 
\alpha^{i-3j-3}
\left( \sum_{\nu=0}^{j} \alpha^{3\nu} \right) 
\lambda_{i-2,j+1,k+1} 
\right) x^{i} y^{j} z^{k+1}, \\
\displaystyle \sum_{i+j+k=d}  
\alpha^{-i} \lambda_{i,j,k} x^i y^{j+1} z^k & = & yX = Xy \\
& = & \displaystyle
\sum_{i+j+k=d}
\left(
\alpha^{-k} \lambda_{i,j,k} + 
\alpha^{-2j-k-3}
\left( \sum_{\nu=0}^{k} \alpha^{3\nu} \right) 
\lambda_{i-2,j+1,k+1} 
\right) x^{i} y^{j+1} z^k .
\end{array}
\]
Thus we have 
\begin{align*}
\textbf{(ii):}
\quad \displaystyle
\left( 1 - \alpha^{i-j}\right) \lambda_{i,j,k} 
&= 
\alpha^{i-3j-3}
\left( \sum_{\nu=0}^{j} \alpha^{3\nu} \right) 
\lambda_{i-2,j+1,k+1}, \\
\quad \displaystyle
\left( 1 - \alpha^{i-k}\right) \lambda_{i,j,k} 
&= 
\alpha^{i-2j-k-3}
\left( \sum_{\nu=0}^{k} \alpha^{3\nu} \right) 
\lambda_{i-2,j+1,k+1} 
\end{align*}
These two equations are equal when $j=k$. 

Next, 
we will calculate $g^{n}=\left(xyz + (1-\alpha^3)^{-1} x^3\right)^n$ 
for each positive integer $n$ with 
$n < |\alpha|$. 
Define
$\kappa^1_{1,1,1}=1$ and $\kappa^1_{3,0,0} = (1-\alpha^3)^{-1}$. 
Then 
$\displaystyle 
\sum_{j=0}^1 \kappa^1_{3-2j, j, j} x^{3-2j} y^j z^j$
is equal to 
$xyz + (1-\alpha^3)^{-1} x^3$. 
Let 
$g^{n}=\displaystyle 
\left(xyz + (1-\alpha^3)^{-1} x^3\right)^n = 
\sum_{j=0}^n \kappa^n_{3n-2j, j, j} x^{3n-2j} y^j z^j$, 
and, 
by induction on $n$, 
let us show that, 
for each $j$, 
\[
\begin{array}{r@{\ }l} 
\left( 1 - \alpha^{3n-3j}\right) 
\kappa^n_{3n-2j, j, j}  & = \displaystyle 
\alpha^{3n-5j-3}
\left( \sum_{\nu=0}^{j} \alpha^{3\nu} \right) 
\kappa^n_{3n-2j-2, j+1, j+1} 
\\
& =  \displaystyle 
\alpha^{3n-5j-3}
(1-\alpha^3)^{-1}(1-\alpha^{3j+3})
\kappa^n_{3n-2j-2, j+1, j+1},\quad\cdots (\ast)
\end{array}
\]
which represents the same formula of {\bf (ii)}. 
Indeed, 
this is true when $n=1$. 
For $n+1$, 
\[
\begin{array}{c@{\ }l} 
& \displaystyle 
\sum_{j=0}^{n+1} \kappa^{n+1}_{3n + 3 -2j, j, j} x^{3n+3-2j} y^j z^j
\\ 
= & \displaystyle 
\left(xyz + (1-\alpha^3)^{-1} x^3\right)^{n+1} 
\\
= & \displaystyle 
\left( \sum_{j=0}^n \kappa^n_{3n-2j, j j} x^{3n-2j} y^j z^j \right)
\left(xyz + (1-\alpha^3)^{-1} x^3\right) 
\\
= & \displaystyle
\sum_{j=0}^n 
\kappa^n_{3n-2j, j, j}
\left( 
x^{3n-2j+1} y^j 
\left( \alpha^{-j} y z^j + 
\alpha^{-j} \left( \sum_{\nu=0}^{j-1} \alpha^{3\nu} \right)  
x^2 z^{j-1}
\right) 
z + 
(1-\alpha^3)^{-1} x^{3n + 3 -2j} y^j z^j
\right) 
\\
= & \displaystyle
\sum_{j=0}^n 
\kappa^n_{3n-2j, j, j}
\left( 
\alpha^{-j} x^{3n-2j+1} y^{j+1} z^{j+1} + 
\alpha^{-3j} (1-\alpha^3)^{-1} x^{3n+3-2j} y^j z^j
\right)
\\
= & \displaystyle
\sum_{j=0}^{n+1} \left( 
\alpha^{-j+1} \kappa^n_{3n+2-2j, j-1, j-1} + 
\alpha^{-3j} (1-\alpha^3)^{-1} \kappa^n_{3n-2j, j, j}
\right)
x^{3n+3-2j} y^j z^j ,
\end{array}
\]
where $\kappa^n_{3n+2, -1, -1} = \kappa^n_{n-2, n+1, n+1}=0$. 
Therefore, 
for each $j$, 
\[
\begin{array}{c@{\ }l} 
\kappa^{n+1}_{3n + 3 -2j, j, j}  & = 
\alpha^{-j+1} \kappa^n_{3n+2-2j, j-1, j-1} + 
\alpha^{-3j} (1-\alpha^3)^{-1} \kappa^n_{3n-2j, j, j} 
\\ & = 
\alpha^{-j+1} 
(1-\alpha^{3n-3j+3})^{-1} \alpha^{3n-5j +2} (1-\alpha^3)^{-1} 
(1-\alpha^{3j}) \kappa^n_{3n-2j, j , j} 
\\ & \hfill + 
\alpha^{-3j} (1-\alpha^3)^{-1} \kappa^n_{3n-2j, j, j} 
\\ & = 
\alpha^{-3j} (1-\alpha^{3n+3}) 
(1-\alpha^3)^{-1} (1-\alpha^{3n-3j+3})^{-1} \kappa^n_{3n-2j, j, j} .
\end{array}
\]
It follows that 
\[
\begin{array}{r@{\ }l} 
& \left( 1 - \alpha^{3n+3-3j}\right) 
\kappa^{n+1}_{3n+3-2j, j, j}  
\\  = & 
\left( 1 - \alpha^{3n+3-3j}\right) 
\alpha^{-3j} (1-\alpha^{3n+3}) 
(1-\alpha^3)^{-1} (1-\alpha^{3n-3j+3})^{-1} \kappa^n_{3n-2j, j, j} 
\\  = & 
\alpha^{-3j} (1-\alpha^{3n+3}) 
(1-\alpha^3)^{-1} 
\left( 1 - \alpha^{3n-3j}\right)^{-1} 
\alpha^{3n-5j-3}
(1-\alpha^3)^{-1}(1-\alpha^{3j+3})
\kappa^n_{3n-2j-2, j+1, j+1} 
\\  = & 
\alpha^{3n-5j} (1-\alpha^3)^{-1} (1-\alpha^{3j+3})
\alpha^{-3j-3} (1-\alpha^{3n+3}) 
(1-\alpha^3)^{-1} \left( 1 - \alpha^{3n-3j}\right)^{-1} 
\kappa^n_{3n-2j-2, j+1, j+1} 
\\ = & \displaystyle 
\alpha^{3n-5j}
(1-\alpha^3)^{-1}(1-\alpha^{3j+3})
\kappa^{n+1}_{3n+1-2j, j+1, j+1} . 
\end{array}
\]

\

\noindent
(1):\,
{\bf Case 1.} 
Suppose that 
$|\alpha|$ is finite and  $|\alpha| = \left| \alpha^3\right|$,
that is, 
$|\alpha|$ cannot be devided by $3$. 
Then 
$\displaystyle 
\sum_{\nu=0}^{|\alpha|-1} \alpha^{3\nu} = 0 $, 
and 
$\displaystyle 
\sum_{\nu=0}^{j} \alpha^{3\nu} \neq 0 $
for each integer $j < |\alpha| -1$. 
Let us show that 
any element 
$X=\displaystyle \sum_{i+j+k=d} \lambda_{i,j,k} x^i y^j z^k$
of the center of $A'$ 
is generated by 
$x^{|\alpha|}$, 
$y^{|\alpha|}$, 
$z^{|\alpha|}$
and 
$g=xyz + (1-\alpha^3)^{-1} x^3$, 
by induction on $d$ and 
the number of non-zero coefficients. 
Suppose that $i_0$ is the largest 
such that 
$\lambda_{i_0,j_0,k_0}$ be a non-zero coefficient 
and  $i_0+j_0+k_0=d$. 
By {\bf (i)}, 
$k_0 - j_0$ is divided by $|\alpha|$. 
Moreover, if both $j_0$ and $k_0$ are non-zero, 
then it follows from {\bf (ii)} 
and the fact that $\lambda_{i_0+2, j_0-1, k_0-1}=0$ 
that 
$\displaystyle 
\sum_{\nu=0}^{j_0-1} \alpha^{3\nu} = 0 $. 
This implies that 
$j_0$ is divided by $|\alpha|$, and hence 
$k_0$ is divided by $|\alpha|$ too. 
Even if $j_0$ or $k_0$ is zero, 
both $j_0$ and $k_0$ are divided by $|\alpha|$. 

By {\bf (ii)}, 
for each $j$, 
\[
\left( 1 - \alpha^{i_0-j_0- 3j}\right) \lambda_{i_0-2j,j_0+j,k_0+j} = 
\alpha^{i_0-3j_0 -3j -3}
\left( \sum_{\nu=0}^{j_0 +j} \alpha^{3\nu} \right) 
\lambda_{i_0-2j -2,j_0+j+1,k_0 + j+1} .
\]
Let $h$ be the non-negative integer less than $|\al|$ such that 
$i_{0}-j_{0}-3h$ is divided by $|\al|$ 
(equivalently, 
$i_0-3h$ is divided by $|\alpha|$)
or 
$\displaystyle 
\sum_{\nu=0}^{j_0 +h} \alpha^{3\nu} = 0$
(equivalently, 
$\displaystyle 
\sum_{\nu=0}^{h} \alpha^{3\nu} = 0$, 
hence $h=|\alpha|-1$). 
By the choice of $h$, 
the term 
$\displaystyle
\sum_{j =0}^{h} 
\lambda_{i_0-2 j ,j_0 + j , k_0 + j}
x^{i_0 - 2 j} y^{j_0+j} z^{k_0 + j} $
is an element of the center of $A'$. 
By inductive hypothesis, 
it suffices to show that 
the term 
$\displaystyle
\sum_{j =0}^{h} 
\lambda_{i_0-2 j ,j_0 + j , k_0 + j}
x^{i_0 - 2 j} y^{j_0+j} z^{k_0 + j} $
is generated by 
$x^{|\alpha|}$, 
$y^{|\alpha|}$, 
$z^{|\alpha|}$
and 
$xyz + (1-\alpha^3)^{-1} x^3$. 

Suppose that 
$i_0$ is divided by $|\alpha|$. 
Then 
by the choice of $h$, $h=0$ holds. 
So the term 
$\displaystyle
\sum_{j =0}^{h} 
\lambda_{i_0-2 j ,j_0 + j , k_0 + j}
x^{i_0 - 2 j} y^{j_0+j} z^{k_0 + j} $
is equal to 
$\lambda_{i_0 ,j_0 , k_0}
x^{i_0 } y^{j_0} z^{k_0}$.
Since all of $i_0$, $j_0$ and $k_0$ are divided by $|\alpha|$, 
$\displaystyle
\sum_{j =0}^{h} 
\lambda_{i_0-2 j ,j_0 + j , k_0 + j}
x^{i_0 - 2 j} y^{j_0+j} z^{k_0 + j}=\lambda_{i_0 ,j_0 , k_0}
x^{i_0 } y^{j_0} z^{k_0}$ is generated by 
$x^{|\alpha|}$, 
$y^{|\alpha|}$, 
$z^{|\alpha|}$. 

Suppose that 
$i_0$ is not divided by $|\alpha|$. 
Then, as seen above ($\ast$), 
the coefficients of the term 
$$
\displaystyle
\sum_{j =0}^{h} 
\lambda_{i_0-2 j ,j_0 + j , k_0 + j}
x^{i_0 - 2 j} y^{j_0+j} z^{k_0 + j}
$$ 
satisfies the same relationship about the coefficients of the term 
$\left( xyz + (1-\alpha^3)^{-1} x^3 \right)^h $. 
Therefore, the term 
$\displaystyle
\sum_{j =0}^{h} 
\lambda_{i_0-2 j ,j_0 + j , k_0 + j}
x^{i_0 - 2 j} y^{j_0+j} z^{k_0 + j} $ 
is equal to 
\[
\gamma
\left( xyz + (1-\alpha^3)^{-1} x^3 \right)^h 
x^{i_0 - 3 h } y^{j_0} z^{k_0}
\]
for some $\gamma \in k$. 
By the choice of $h$ and {\bf (ii)}, 
$i_0-3h$, $j_0$ and $k_0$ 
are divided by $|\alpha|$. 
Therefore, $\displaystyle
\sum_{j =0}^{h} 
\lambda_{i_0-2 j ,j_0 + j , k_0 + j}
x^{i_0 - 2 j} y^{j_0+j} z^{k_0 + j} $ 
is generated by 
$x^{|\alpha|}$, 
$y^{|\alpha|}$, 
$z^{|\alpha|}$
and 
$\displaystyle
xyz + (1-\alpha^3)^{-1} x^3$. 

\noindent
{\bf Case 2.} 
Suppose that 
$|\alpha|$ is finite and 
$|\alpha| \neq \left| \alpha^3\right|$. 
Then $|\alpha| = 3 \left| \alpha^3\right|$ holds, 
so $\displaystyle 
\sum_{\nu=0}^{|\alpha^3|-1} \alpha^{3\nu} = 0 $, 
and for each integer $j < \left|\alpha^3\right| -1$, 
$\displaystyle 
\sum_{\nu=0}^{j} \alpha^{3\nu} \neq 0 $, 
also, the following equations hold; 
\[
\begin{array}{lll}
y x^{|\alpha^3|} = \alpha^{-|\alpha^3|} x^{|\alpha^3|} y ,
& 
z x^{|\alpha^3|} = \alpha^{|\alpha^3|} x^{|\alpha^3|} z ,
& 
y^{|\alpha^3|} x = \alpha^{-|\alpha^3|} x^{|\alpha^3|} y ,
\\[10pt]
z^{|\alpha^3|} x = \alpha^{|\alpha^3|} x^{|\alpha^3|} z ,
& 
z y^{|\alpha^3|} = \alpha^{-|\alpha^3|} y^{|\alpha^3|} z ,
&
z^{|\alpha^3|} y = \alpha^{-|\alpha^3|} y z^{|\alpha^3|} . 
\end{array}
\]
So $x^{|\alpha^3|}y^{|\alpha^3|}z^{|\alpha^3|}$
is an element of the center of $A$. 
As seen before, 
for each $n < |\alpha^3|$, 
when $\displaystyle 
\left(xyz + (1-\alpha^3)^{-1} x^3\right)^n $ 
is denoted by 
$\sum_{j=0}^n \kappa^n_{3n-2j, j, j} x^{3n-2j} y^j z^j$, 
for each $j$, 
\[
\begin{array}{r@{\ }l} 
\left( 1 - \alpha^{3n-3j}\right) 
\kappa^n_{3n-2j, j, j}  & = \displaystyle 
\alpha^{3n-5j-3}
\left( \sum_{\nu=0}^{j} \alpha^{3\nu} \right) 
\kappa^n_{3n-2j-2, j+1, j+1} 
\\
& =  \displaystyle 
\alpha^{3n-5j-3}
(1-\alpha^3)^{-1}(1-\alpha^{3j+3})
\kappa^n_{3n-2j-2, j+1, j+1} ,
\end{array}
\]
which represents the same formula of {\bf (ii)}. 
Hence, 
when $n=|\alpha^3|$, 
for each integer $j$ with $1\leq j <|\alpha^3|$, 
$\kappa^{|\alpha^3|}_{3|\alpha^3|-2j,j,j}=0$ holds. 
Moreover, 
by induction on $n$, 
it is proved that 
$\kappa^n_{n,n,n}=\alpha^{-\frac{(n-1)n}{2}}$
and 
$\kappa^n_{3n,0,0} = (1-\alpha^3)^{-n}$. 
Therefore, 
\[
\left(xyz + (1-\alpha^3)^{-1} x^3\right)^{|\alpha^3|} = 
\alpha^{-\frac{(|\alpha^3|-1) |\alpha^3| }{2}} 
x^{|\alpha^3|}y^{|\alpha^3|}z^{|\alpha^3|} + 
(1-\alpha^3)^{-|\alpha^3|} x^{3|\alpha^3|}.
\]
Hence 
\[
x^{|\alpha^3|}y^{|\alpha^3|}z^{|\alpha^3|} = 
\alpha^{\frac{(|\alpha^3|-1) |\alpha^3| }{2}} 
\left(xyz + (1-\alpha^3)^{-1} x^3\right)^{|\alpha^3|} - 
\alpha^{\frac{(|\alpha^3|-1) |\alpha^3| }{2}} 
(1-\alpha^3)^{-|\alpha^3|} x^{|\alpha|} .
\]

Let us show that 
any element 
$X=\displaystyle \sum_{i+j+k=d} \lambda_{i,j,k} x^i y^j z^k$
of the center of $A'$ 
is generated by 
$x^{|\alpha|}$, 
$y^{|\alpha|}$, 
$z^{|\alpha|}$ 
and 
$xyz + (1-\alpha^3)^{-1} x^3$, 
by induction on $d$ and 
the number of non-zero coefficients. 
Suppose that $i_0$ is the largest 
such that 
$\lambda_{i_0,j_0,k_0}$ be a non-zero coefficient 
and  $i_0+j_0+k_0=d$. 
By {\bf (i)}, 
$k_0 - j_0$ is divided by $3|\alpha^3|$ ($=|\alpha|$). 
Moreover, if both $j_0$ and $k_0$ are non-zero, 
then it follows from the first equation of {\bf (ii)} 
and the fact that $\lambda_{i_0+2, j_0-1, k_0-1}=0$ 
that 
$\displaystyle 
\sum_{\nu=0}^{j_0-1} \alpha^{3\nu} = 0 $,
which implies that 
$j_0$ is divided by $|\alpha^3|$. 
Similarly, 
by  the second equation of {\bf (ii)}, 
it follows that 
$k_0$ is divided by $|\alpha^3|$. 
Even if $j_0$ or $k_0$ is zero, 
both $j_0$ and $k_0$ are divided by $|\alpha^3|$. 
Let 
$j_0 =  |\alpha^3|(3j_1+j_2)$, 
and 
$k_0 = |\alpha^3| (3k_1+j_2)$, 
where 
$0\leq  j_2 \leq 2$.

Let $h$ be the non-negative integer less than $|\al|$ such that 
$i_{0}-j_{0}-3h$ is divided by $3|\al^{3}|$ 
(equivalently, 
$i_0-3h$ is divided by $|\alpha^{3}|$)
or 
$\displaystyle 
\sum_{\nu=0}^{j_0 +h} \alpha^{3\nu} = 0$
(equivalently, 
$\displaystyle 
\sum_{\nu=0}^{h} \alpha^{3\nu} = 0$, 
hence $h=|\alpha^{3}|-1$). 
Then, 
by the choice of $h$, the term 
$\displaystyle
\sum_{j =0}^{h} 
\lambda_{i_0-2 j ,j_0 + j , k_0 + j}
x^{i_0 - 2 j} y^{j_0+j} z^{k_0 + j} $
is an element of the center of $A'$. 
By inductive hypothesis, 
it suffices to show that 
the term 
$\displaystyle
\sum_{j =0}^{h} 
\lambda_{i_0-2 j ,j_0 + j , k_0 + j}
x^{i_0 - 2 j} y^{j_0+j} z^{k_0 + j} $
is generated by 
$x^{|\alpha|}$, 
$y^{|\alpha|}$, 
$z^{|\alpha|}$, 
$x^{|\alpha^3|}y^{|\alpha^3|}z^{|\alpha^3|}$ 
and 
$g=xyz + (1-\alpha^3)^{-1} x^3$.

Suppose that 
$i_0 - j_0 $ is divided by $3|\alpha^3|$. 
Then 
by the choice of $h$, 
$h=0$
and hence 
the term 
$\displaystyle
\sum_{j =0}^{h} 
\lambda_{i_0-2 j ,j_0 + j , k_0 + j}
x^{i_0 - 2 j} y^{j_0+j} z^{k_0 + j} $
is equal to 
$\lambda_{i_0 ,j_0 , k_0}
x^{i_0 } y^{j_0} z^{k_0}$. 
Let 
$i_0 = |\alpha^3| (3i_1+i_2)$, 
where 
$0\leq i_2 \leq 2$.
Since $i_0 - j_0 $ is divided by $3|\alpha^3|$, 
$i_2=j_2$.
Since 
$x^{|\alpha|}$, 
$y^{|\alpha|}$, 
$z^{|\alpha|}$
and 
$x^{|\alpha^3|}y^{|\alpha^3|}z^{|\alpha^3|}$ 
are elements of the center of $|\alpha|$, 
and $|\alpha|=3|\alpha^3|$, 
it follows that, 
for some $\gamma\in k$, 
\[
\lambda_{i_0 ,j_0 , k_0}
x^{i_0 } y^{j_0} z^{k_0} = 
\gamma 
\left( x^{|\alpha|}\right)^{i_1} 
\left( y^{|\alpha|} \right)^{j_1}
\left( z^{|\alpha|}\right)^{k_1}
\left( x^{|\alpha^3|}y^{|\alpha^3|}z^{|\alpha^3|} \right)^{i_2} ,
\]
which is generated by 
$x^{|\alpha|}$, 
$y^{|\alpha|}$, 
$z^{|\alpha|}$
and 
$x^{|\alpha^3|}y^{|\alpha^3|}z^{|\alpha^3|}$, 
so
is generated by 
$x^{|\alpha|}$, 
$y^{|\alpha|}$, 
$z^{|\alpha|}$ 
and 
$xyz + (1-\alpha^3)^{-1} x^3$.

Suppose that 
$i_0- j_0$ is not divided by $3|\alpha^3|$
and 
$i_0-j_0-3h$ is divided by $3|\alpha^3|$. 
Since $j_0$ is divided by $|\alpha^3|$, 
$i_0 - 3 h $ is divided by $|\alpha^3|$. 
Since 
$j_0 =  |\alpha^3|(3j_1+j_2)$, 
there exists a non-negative integer  $i_1$
such that 
$i_0 - 3 h = |\alpha^3|(3i_1 + j_2)$. 
Since 
$x^{|\alpha|}$, 
$y^{|\alpha|}$, 
$z^{|\alpha|}$, 
$x^{|\alpha^3|}y^{|\alpha^3|}z^{|\alpha^3|}$ 
and 
$xyz + (1-\alpha^3)^{-1} x^3$
are elements of the center of $A'$, 
as in the previous case, 
there exists $\gamma \in K$ 
such that 
$\displaystyle 
\sum_{j =0}^{h} 
\lambda_{i_0-2 j , j_0+j , k_0 + j} 
x^{i_0 - 2 j} y^{j_0 + j} z^{k_0 + j} $
is equal to 
\[
\gamma 
\left(
xyz + (1-\alpha^3)^{-1} x^3
\right)^h
x^{|\alpha|i_1}
y^{|\alpha|j_1} 
z^{|\alpha|k_1}
\left( x^{|\alpha^3|}y^{|\alpha^3|}z^{|\alpha^3|} \right)^{i_2} ,
\]
which 
is generated by 
$x^{|\alpha|}$, 
$y^{|\alpha|}$, 
$z^{|\alpha|}$, 
$x^{|\alpha^3|}y^{|\alpha^3|}z^{|\alpha^3|} $
and 
$xyz + (1-\alpha^3)^{-1} x^3$.

Suppose that 
$i_0-j_0-3j$ is not divided by $3|\alpha^3|$ 
for every $j<|\alpha^3|-1$. 
Then $h=|\alpha^{3}|-1$, hence
$\displaystyle 
\sum_{\nu=0}^{h} \alpha^{3\nu} = 0$, 
and 
by {\bf (ii)}, 
$\lambda_{i_0-2h, j_0+h, k_0+h}\neq 0$. 
Therefore, by {\bf (ii)}, 
$i_0 - j_0 -3h$ is divided by $3|\alpha^3|$. 
So as in the previous paragraph, 
$\displaystyle 
\sum_{j =0}^{h} 
\lambda_{i_0-2 j , j_0+j , k_0 + j} 
x^{i_0 - 2 j} y^{j_0 + j} z^{k_0 + j} $
is generated by 
$x^{|\alpha|}$, 
$y^{|\alpha|}$, 
$z^{|\alpha|}$, 
$x^{|\alpha^3|}y^{|\alpha^3|}z^{|\alpha^3|} $
and 
$xyz + (1-\alpha^3)^{-1} x^3$, 
so
is generated by 
$x^{|\alpha|}$, 
$y^{|\alpha|}$, 
$z^{|\alpha|}$ 
and 
$xyz + (1-\alpha^3)^{-1} x^3$. 

\

\noindent
(2):\,
{\bf Case 3.} 
Suppose that $|\alpha|$ is infinite.
Let us show that any element 
$X=\displaystyle \sum_{i+j+k=d} \lambda_{i,j,k} x^i y^j z^k$
of the center of $A'$ 
is generated by 
$g=xyz + (1-\alpha^3)^{-1} x^3$, 
by induction on $d$ and 
the number of non-zero coefficients. 
By {\bf (i)}, 
$\lambda_{i,j,k}$ have to be $0$ if $j\neq k$. 
Suppose that $i_0$ is the largest 
such that 
$\lambda_{i_0,j_0,j_0}$ be a non-zero coefficient 
and  $i_0+2 j_0 =d$. 
If $j_0$ is non-zero, 
then it follows from {\bf (ii)} 
and the fact that $\lambda_{i_0+2, j_0-1, j_0-1}=0$ 
that 
$\displaystyle 
\alpha^{i_0 - 3 j_0 +2}
\sum_{\nu=0}^{j_0-1} \alpha^{3\nu} = 0 $,
which is a contradiction because 
$|\alpha|$ is infinite.
Thus $j_0=0$. 
It follows that 
$X=\displaystyle \sum_{j=0}^{h} \lambda_{i_0-2j,j,j} x^{i_0-2j} y^j z^j$, 
where 
$h$ is the integer such that 
$i_0-3h \in \{0,1,2\}$. 
Then,
the coefficients of the term 
$\displaystyle
\sum_{j =0}^{h} 
\lambda_{i_0-2 j ,j , j}
x^{i_0 - 2 j} y^{j} z^{ j} $ 
satisfies the same relationship 
about 
the coefficients of the term 
$\left( xyz + (1-\alpha^3)^{-1} x^3 \right)^h $, 
hence 
\[ X = 
\gamma
\left( xyz + (1-\alpha^3)^{-1} x^3 \right)^h 
x^{i_0 - 3 h } 
\]
for some $\gamma \in k$. 
Since $X$ and $xyz + (1-\alpha^3)^{-1} x^3 $ are elements of 
the center of $A'$, 
and $|\alpha|$ is infinite, 
$i_0 = 3h$ holds. 
Therefore, $X$ is generated by 
$g=\displaystyle
xyz + (1-\alpha^3)^{-1} x^3$. 
\end{proof}
\section{Main Results}

Finally, in this section, we prove that the conjecture in Remark \ref{rem-IMo1} hold 
for a Type S' algebra. 
In order to prove this, we need to prove the same statements of Theorem \ref{Mo2-thm4.17} and Theorem \ref{Mo2-thm4.21}
in Subsection 2.2. 
For a $3$-dimensional quantum polynomial algebra of Type S', 
we remark the following: 
\begin{rem}
\label{rem-TypeS'}
Let $A=\mathcal{A}(E,\si)=k\langle x,y,z\rangle/(f_{1},f_{2},f_{3})$
be a $3$-dimensional quantum polynomial algebra of Type S' 
where 
$$\left\{
\begin{array}{ll}
f_{1}=yz-\al zy+x^{2},\\
f_{2}=zx-\be xz,\\
f_{3}=xy-\be yx\quad (\al,\be\in k,\,\al\be^{2}\neq 0,1)
\end{array}
\right.$$
(see \cite[Theorem 3.2]{IMa1}, \cite[Table 1 in Proposition 3.1]{IMa2}). 
For a $3$-dimensional quantum polynomial algebra $A=\cA(E,\si)$ of Type S', 
there exists the $3$-dimensional Calabi-Yau quantum polynomial algebra $A'$ of Type S' 
such that $\grmod\,A\cong \grmod\,A'$ 
so that $\Projn A\cong \Projn A'$ 
by Lemma \ref{IMa-Main2}, 
where $A'=\mathcal{A}(E,\si')=k\langle x,y,z\rangle/(g_{1},g_{2},g_{3})$; 
$$\left\{
\begin{array}{ll}
g_{1}=yz-\alpha zy+x^{2},\\
g_{2}=zx-\alpha xz,\\
g_{3}=xy-\alpha yx\quad (\alpha^{3}\neq 0,1)
\end{array}
\right.$$
(see \cite[Table 2 in Theorem 3.4]{IMa2}).  
\end{rem}

The same statement of Theorem \ref{Mo2-thm4.17} holds for a Type S' algebra: 
\begin{thm}
\label{thm1}
Let $A=\cA(E,\si)$ be a $3$-dimensional quantum polynomial algebra of Type S'. 
Then the following are equivalent: 
\begin{enumerate}[{\rm (1)}]
\item $\|\si\|<\infty$. 
\item $\Projn A$ is finite over its center. 
\item The Beilinson algebra $\nabla A$ of $A$ is $2$-representation tame. 
\end{enumerate}
\end{thm}
\begin{proof}
Let $A$ be a $3$-dimensional quantum polynomial algebra of Type S' and 
$A'$ a $3$-dimensional Calabi-Yau quantum polynomial algebra of Type S' as in Remark \ref{rem-TypeS'}. 

Considering that the Nakayama automorphism $\nu$ of $A'$ is identity because $A'$ is Calabi-Yau, 
by \cite[Lemma 4.16 (4), Theorem 1.6 (2)]{Mo2}, 
we have 
$$
\Pi(\nabla\,A)\cong \Pi(\nabla\,A')\cong (A')^{[3]}
$$
where $\Pi(\nabla\,A)$ (respectively, $\Pi(\nabla A')$) is the preprojective algebra of $\nabla\,A$ 
(respectively, $\nabla\,A'$), 
and $(A')^{[3]}$ is the $3$rd quasi-Veronese algebra of $A'$. 

(1)$\,\Longleftrightarrow\,$(2): It follows from Lemma \ref{q.nu}. 

(1)$\,\Longrightarrow\,$(3): 
If $\|\sigma\|<\infty$, then $|\si'|(=|\al^{3}|)=\|\si'\|=\|\si\|<\infty$ 
by Lemma \ref{q-main} and Proposition \ref{lem_Mo2} (2). 
So, $A'$ is finite over $Z(A')$ by Theorem \ref{thm_{ATV2}}. 
It follows from Proposition \ref{prop-S'} (1) 
that $Z(A')=k[x^{|\alpha|},y^{|\alpha|},z^{|\alpha|},g]$ is noetherian 
where $g=xyz + (1-\alpha^3)^{-1} x^3 \in A_{3}$. 
We remark that $|\si'|=|\al^{3}|<\infty\,\Longleftrightarrow\, |\al|<\infty$. 
Therefore, by Lemma \ref{Mo2-prop4.11}, 
$(A')^{[3]}$ is right noetherian and finite over its center. 
It follows from $(A')^{[3]}\cong \Pi(\nabla\,A')\cong \Pi(\nabla\,A)$ 
that $\Pi(\nabla\,A)$ is right noetherian and finite over its center. 
So, by Definition \ref{def-d-rep-tame}, $\nabla\,A$ is $2$-representation tame. 

(3)$\,\Longrightarrow\,$(1): 
if $\|\si\|=\infty$, then $\|\si'\|=|\si'|=|\al|=\infty$, 
so, by Proposition \ref{prop-S'} (2),  $Z(A')=k[g]$, 
where $g=xyz + (1-\alpha^3)^{-1} x^3 \in A_{3}$. 
By Lemma \ref{Mo2-lem4.10}, 
$$
Z(\Pi(\nabla\,A'))=Z((A')^{[3]})=Z(A')^{[3]}=k[g]^{(3)}=k[w]
$$
with ${\rm deg}\,w=1$, 
where $Z(A')^{[3]}$ is the $3$rd quasi-Veronese algebra of $Z(A')^{[3]}$ 
and $k[g]^{(3)}$ is the $3$rd quasi-Veronese algebra of $k[g]$. 
Note that $\mathrm{GKdim}\,\Pi(\nabla\,A')=\mathrm{GKdim}\,A'=3$. 
On the other hand, 
$\mathrm{GKdim}\,Z(\Pi(\nabla\,A'))=\mathrm{GKdim}\,k[w]=1$. 
Therefore, $\Pi(\nabla\,A')\cong \Pi(\nabla\,A)$ is not finite over its center. 
By Definition \ref{def-d-rep-tame}, $\nabla\,A$ is not $2$-representation tame. 
\end{proof}

The same statement of Theorem \ref{Mo2-thm4.21} holds for a Type S' algebra: 
\begin{thm}
\label{thm2}
Let $A=\cA(E,\si)$ be a $3$-dimensional quantum polynomial algebra of Type S'. 
\begin{enumerate}[{\rm (1)}]
\item If the Beilinson algebra $\nabla A$ of $A$ is not $2$-representation tame, 
then the isomorphism classes of simple $2$-regular modules over $\nabla A$ 
are parametrized by the set of points of $E\subsetneq \PP^{2}$. 
\item If the Beilinson algebra $\nabla A$ of $A$ is $2$-representation tame, 
then the isomorphism classes of simple $2$-regular modules over $\nabla A$ 
are parametrized by the set of points of $\PP^{2}$. 
\end{enumerate}
\end{thm}
\begin{proof}
Let $A$ be a $3$-dimensional quantum polynomial algebra of Type S' and 
$A'$ a $3$-dimensional Calabi-Yau quantum polynomial algebra of Type S' as in Remark \ref{rem-TypeS'}. 

(1): If $\nabla\,A$ is not $2$-representation tame, then 
$\|\si\|=\infty$ by Theorem \ref{thm1}. 
Therefore, by Lemma \ref{Mo2-thm4.8}
the isomorphism classes of simple $2$-regular modules over $\nabla A$ 
are parametrized by the set of points of $E\subsetneq \PP^{2}$. 

(2):  If $\nabla\,A$ is $2$-representation tame, then 
$\|\si\|(=|\al^{3}|)<\infty$ by Theorem \ref{thm1}. 
(Note that $|\al|:=l<\infty$ holds. )
Since $\Projn A\cong \Projn A'$ holds, $|\Projn A|=|\Projn A'|$ also holds. 
We easily check that $x\in A'_{1}$ is normal element. 
It follows from Lemma \ref{Mo2-thm4.20} that 
$$
|\Projn A|=|\Projn A'|=|\Projn\, A'/(x)|\,\bigsqcup\,|\Specn\, A'[x^{-1}]_{0}|.
$$
Since $A'/(x)\cong k\langle y,z \rangle/(yz-\al zy)$, 
which is $2$-dimensional AS-regular algebra 
(\cite[page 172]{AS}), 
we have $|\Projn\, A'/(x)|=|\PP^{1}|$. 
Also, we claim that 
$A'[x^{-1}]_{0}\cong k\langle u,v\rangle/(uv-\al^{3} vu+\al)
$, 
where $\al^{3}\neq 0,1$, $u:=yx^{-1}$, $v:=zx^{-1}$ and 
$k\langle u,w\rangle/(uw-\al wu-1)$ is called a \textit{quantized Weyl algebra}.
Indeed, 
$$
\begin{cases}
x^{-1}g_{1}x^{-1}=x^{-1}yzx^{-1}-\al x^{-1}zyx^{-1}+1,\\
x^{-1}g_{2}x^{-1}=x^{-1}z-\al zx^{-1}, \\
x^{-1}g_{3}x^{-1}=yx^{-1}-\al x^{-1}y.
\end{cases}
$$
So, we have the following relation; 
\begin{align*}
uv&=(yx^{-1})(zx^{-1})=yx^{-1}zx^{-1}\\
&=\al x^{-1}yzx^{-1}=\al x^{-1}(\al zy-x^{2})x^{-1}\\
&=\al^{2}x^{-1}zyx^{-1}-\al=\al^{3}zx^{-1}yx^{-1}-\al\\
&=\al^{3}vu-\al.
\end{align*}
(Dividing the both side of this relation $uv-\al^{3}vu+\al=0$ by $-\al(\neq 0)$, 
$-\dfrac{1}{\al}uv+\al^{2}vu-1=0$. 
Next, putting $w:=-\al v$ in this equation, 
we obtain the equation $uw-\al wu-1=0$. )
Since $A'[x^{-1}]_{0}\cong k\langle u,w\rangle/(uw-\al wu-1)$, 
the following equation holds; 
$$|\Specn\, A'[x^{-1}]_{0}|=|\Specn\, K\langle u,w\rangle/(uw-\al wu-1)|=|\mathbb{A}^{2}|.$$
Indeed, 
by \cite[Proposition 2.6]{HW} (\cite[Theorem 5.6]{DGO}), 
any irreducible matrix solution $(A,B)$ of the equation $wu-\al uw=1$, 
in which $BA\neq AB$, is equivalent to either the solution 
$(U_{\la},\,W_{\la,\eta})$ or $(U,\,W_{\beta})$, 
where $U_{\la}$, $\,W_{\la,\eta}$, $U$ and $W_{\beta}$ 
are matrixes $l\times l$ 
for $\la\in k\backslash\{0\}/\langle \al \rangle$, $\eta\in k\backslash\{0\}$, 
$\beta\in k$ as follows;
$$
U_{\la}=\la
\begin{pmatrix}
\al & 0 & \cdots & 0\\
0 & \al^{2} & \cdots & 0\\
\vdots & \vdots & \ddots & \vdots\\
0 & 0 & \cdots & \al^{l}
\end{pmatrix},\,
W_{\la,\eta}=
\begin{pmatrix}
\frac{1}{(1-\al)\al\la} & 1 & 0 &\cdots & 0 & 0\\
0 & \frac{1}{(1-\al)\al^{2}\la}& 1 & \cdots & 0 & 0 \\
\vdots & \vdots & \vdots & \ddots & \vdots & \vdots \\
0 & 0 & 0 & \cdots & \frac{1}{(1-\al)\al^{l-2}\la} & 1 \\
\eta & 0 & 0 & \cdots & 0 &\frac{1}{(1-\al)\al^{l}\la}
\end{pmatrix},
$$
$$
U=
\begin{pmatrix}
0 & 1 & 0 & \cdots & 0 \\
  & 0 & 1 & \cdots & 0 \\
  &   &   & \ddots & 0 \\
  &   &   &  0     & 1 \\
  &   &   &        & 0  
\end{pmatrix},\,
W_{\beta}=
\begin{pmatrix}
0 & 0 & \cdots & 0 & 0 & \beta \\
\sum_{i=0}^{l-2} \al^{i}& 0 & \cdots & 0 & 0 & 0 \\
0 & \sum_{i=0}^{l-3} \al^{i} & \cdots & 0 & 0 & 0 \\
\vdots & \vdots & \ddots & \vdots & \vdots & \vdots \\
0 & 0 & \cdots & 1+\al & 0 & 0 \\
0 & 0 & \cdots & 0 & 1 & 0
\end{pmatrix}.
$$
So, the set of irreducible solution of the equation $wu-\al uw=1$ 
is corresponding to 
$$
X:=\{(U_{\la},\,W_{\la,\eta}) \mid \la\in k\backslash\{0\}/\langle \al \rangle,\,\eta\in k\backslash\{0\}\}\cup
\{(U,\,W_{\beta})\mid \beta\in k\}.
$$
Here, we consider the following map 
$$
\mathbb{A}^{2}=\{(\la,\eta)\mid \la,\,\eta\in k\}
\,\longrightarrow X;\,
(\la,\eta)\,\longmapsto\,
\begin{cases}
(U_{\la},\,W_{\la,\eta})\quad\text{if $\la\in k\backslash\{0\}/\langle \al \rangle$ and $\eta\in k\backslash\{0\}$}, \\
(U,\,W_{\la})\quad\text{if $\eta=0$}.
\end{cases}
$$
Then, the above map is a bijection,  
so $|\Specn\, k\langle u,w\rangle/(uw-\al wu-1)|=|\mathbb{A}^{2}|$ holds. 
Therefore, 
$$
|\Projn A|=|\Projn A'|=|\PP^{1}|\,\bigsqcup\,|\mathbb{A}^{2}|=|\PP^{2}|.
$$
\end{proof}

By Theorem \ref{thm1} and Theorem \ref{thm2}, 
the conjecture in Remark \ref{rem-IMo1} holds for Type S':  
\begin{thm}
\label{Main-I}
For a $3$-dimensional quantum polynomial algebra $A$ of Type S', 
the following are equivalent: 
\begin{enumerate}[{\rm (1)}]
\item $\Projn A$ is finite over its center. 
\item The Beilinson algebra $\nabla A$ of $A$ is $2$-representation tame. 
\item The isomorphism classes of simple $2$-regular modules over $\nabla A$ are parameterized by $\PP^2$.
\end{enumerate}
\end{thm}

\proof[Acknowledgements]
The author was supported by 
Grants-in-Aid for Young Scientific Research 21K13781 Japan Society for the Promotion of Science. 



\begin{thebibliography}{HD}

\bibitem[AOU]{AOU}
T. ~Abdelgadir, S. ~Okawa and K.~Ueda, 
Compact moduli of noncommutative projective planes, 
preprint (arXiv:1411.7770).


 \bibitem[AS]{AS}
  M.~Artin and W.~Schelter, 
  Graded algebras of global dimension $3$, 
   \textit{Adv. Math.} \textbf{66} (1987), 171--216. 


\bibitem[ATV1]{ATV1}
   M.~Artin, J.~Tate and M.~Van~den~Bergh, 
   Some algebras associated to automorphisms of elliptic curves, 
   The Grothendieck Festschrift, vol. 1, 
   \textit{Progress in Mathematics} vol. \textbf{86} (Birkh\"auser, Basel, 1990) 33--85. 

\bibitem[ATV2]{ATV2}
  M.~Artin, J.~Tate and M.~Van~den~Bergh, 
   Module over regular algebras of dimension $3$, 
   \textit{Invent. Math.} \textbf{106} (1991), no. 2, 335--388. 
%
\bibitem[AZ]{AZ} 
  M.~Artin and J.~J.~Zhang, 
  Noncommutative projective schemes, 
  \textit{Adv. Math.} \textbf{109} (1994), no. 2, 228--287.

\bibitem[DGO]{DGO}
Y. A. Drozd, B. L. Guzner and S. A. Ovsienko, 
Weight modules over generalized Weyl algebras, 
\textit{J. Algebra}, \textbf{184} (1996), no. 2, 491--504. 


\bibitem[HIO]{HIO}
M.~Herschend, O.~Iyama and S.~Oppermann, 
$n$-representation infinite algebras, 
\textit{Adv. Math.} \textbf{252} (2014), 292--342. 

\bibitem[HW]{HW}
B.~Heider and L.~Wang, 
Irreducible representations of the quantum Weyl algebra at roots of unity given by matrices, 
\textit{Comm. Algebra}, \textbf{42} (2014), no. 5, 2156--2162.

\bibitem[IMa1]{IMa1}
A.~Itaba and M.~Matsuno, 
Defining relations of $3$-dimensional quadratic AS-regular algebras, 
\textit{Math. J. Okayama Univ.} \textbf{63} (2021), 61--86. 

\bibitem[IMa2]{IMa2}
A.~Itaba and M.~Matsuno, 
AS-regularity of geometric algebras of plane cubic curves, 
\textit{J. Aust. Math. Soc.}, 
\textbf{112} (2022), Issue 2, 193-217. 

\bibitem[IMo]{IMo}
A.~Itaba and  I.~ Mori, 
 Quantum projective planes finite over their centers,
 \textit{Canad. Math. Bull.}, Vol. \textbf{66} (2023), Issue 1, 53--67. 

\bibitem[MiMo]{MM}
H.~Minamoto and I.~ Mori,
The structure of AS-Gorenstein algebras,
\textit{Adv. Math.} \textbf{226} (2011), no. 5, 4061--4095. 

\bibitem[Mo1]{Mo1}
I.~Mori, 
The center of some quantum projective planes, 
\textit{J. Algebra} \textbf{204} (1998), no. 1, 15--31. 

\bibitem[Mo2]{Mo}
I.~Mori,
Non commutative projective schemes and point schemes, 
\textit{Algebras, Rings and Their Representations}, 
World Sci., Hackensack, N.J., (2006), 215--239. 

\bibitem[Mo3]{Mo3}

I.~Mori, B-construction and C-construction, 
\textit{Comm. Algebra} \textbf{41} (2013), no. 6, 2071--2091. 

\bibitem[Mo4]{Mo2}
I.~Mori, 
Regular modules over $2$-dimensional quantum Beilinson algebras of Type S, 
\textit{Math. Z.} \textbf{279} (2015), no. 3--4, 1143--1174. 

\bibitem[MU]{MU1} 
I.~Mori and K.~Ueyama, 
Graded Morita equivalences for geometric AS-regular algebras, 
\textit{Glasg. Math. J.} \textbf{55} (2013), no. 2, 241--257.


\end{thebibliography}
\end{document}